\documentclass[12pt]{amsart}
\usepackage[utf8]{inputenc}
\usepackage[english]{babel}
\usepackage{amsmath}
\usepackage{amssymb}
\usepackage{amscd}
\usepackage{amsthm}
\usepackage{amscd}
\usepackage{tikz-cd}

\usepackage{changebar,color}

\usepackage{verbatim, url}
\DeclareMathOperator{\Cr}{\mathrm{Cr}}

\DeclareMathOperator{\PGL}{\mathrm{PGL}}

\DeclareMathOperator{\GL}{\mathrm{GL}}

\DeclareMathOperator{\Aut}{\mathrm{Aut}}

\DeclareMathOperator{\Char}{\mathrm{char}}

\begin{document}

\newcommand{\Sm}{\mathcal{S}^{\!}\mathsf{m}_k}
\newcommand{\Rings}{\mathcal{R}^{\!}\mathsf{ings}^*}
\newcommand{\rings}{\mathcal{R}^{\!}\mathsf{ings}}
\newcommand{\SmE}[1]{\mathcal{S}^{\!}\mathsf{m}_{#1}}
\newcommand{\Mot}[1]{\mathcal M^{\!}\mathsf{ot}_{#1}}
\newcommand{\Corr}[1]{\mathcal C^{\!}\mathsf{orr}_{#1}}
\newcommand{\MotF}[1]{\mathcal M_{#1}}
\newcommand{\QG}{\mathbb Q\Gamma}
\newcommand{\Inv}{\mathrm{Inv}(\QG)}
\newcommand{\Gal}{\mathrm{Gal}}
\newcommand{\End}{\mathrm{End}}
\newcommand{\M}[1]{\mathcal{M}_{#1}}
\newcommand{\EG}{\!\,_EG}
\newcommand{\EGP}{\!\,_E(G/P)}
\newcommand{\EX}{\!\,_EX}
\newcommand{\XG}{\!\,_{\xi}G}
\newcommand{\XGP}{\!\,_{\xi}(G/P)}
\newcommand{\KQ}[1]{\mathrm K(n)^*\big(#1;\,\mathbb Q[v_n^{\pm1}]\big)}
\newcommand{\KZ}[1]{\mathrm K(n)^*\big(#1;\,\mathbb Z_{(p)}[v_n^{\pm1}]\big)}
\newcommand{\KZp}[1]{\mathrm K(n)^*\big(#1;\,\mathbb Z_p[v_n^{\pm1}]\big)}
\newcommand{\KF}[1]{\mathrm K(n)^*\big(#1;\,\mathbb F_p[v_n^{\pm1}]\big)}
\newcommand{\CHQ}[1]{\mathrm{CH}^*\big(#1;\,\mathbb Q[v_n^{\pm1}]\big)}
\newcommand{\KXZ}[1]{\!\,^{\mathrm K(n)\!}#1_{\,\mathbb Z_{(p)}[v_n^{\pm1}]}}
\newcommand{\KMotQ}{\Mot{\,\mathrm K(n)}}
\newcommand{\CHMotQv}{\Mot{\,\mathrm{CH}}}
\newcommand{\KXQ}[1]{\mathcal M_{\mathrm K(n)}(#1)}
\newcommand{\KMQ}{\mathcal M_{\,\mathrm K(n)}}
\newcommand{\CHMQv}{\MotF{\,\mathrm{CH}}}
\newcommand{\KCorrQ}{\Corr{\,\mathrm K(n)}}
\newcommand{\CHCorrQv}{\Corr{\,\mathrm{CH}}}
\newcommand{\CHCorrQ}{\Corr{\,\mathrm{CH}}}
\newcommand{\AMot}{\Mot A}

\newcommand{\e}{\varepsilon}
\newcommand{\con}{\ensuremath{\triangledown}}
\newcommand{\ra}{\ensuremath{\rightarrow}}
\newcommand{\tp}{\ensuremath{\otimes}}
\newcommand{\pr}{\ensuremath{\partial}}
\newcommand{\trigd}{\ensuremath{\triangledown}}
\newcommand{\dAB}{\ensuremath{\Omega_{A/B}}}
\newcommand{\QQ}{\ensuremath{\mathbb{Q}}\xspace}
\newcommand{\CC}{\ensuremath{\mathbb{C}}\xspace}
\newcommand{\RR}{\ensuremath{\mathbb{R}}\xspace}
\newcommand{\ZZ}{\ensuremath{\mathbb{Z}}\xspace}
\newcommand{\Zp}{\ensuremath{\mathbb{Z}_{(p)}}\xspace}
\newcommand{\Z}[1]{\ensuremath{\mathbb{Z}_{(#1)}}\xspace}
\newcommand{\NN}{\ensuremath{\mathbb{N}}\xspace}
\newcommand{\LL}{\ensuremath{\mathbb{L}}\xspace}
\newcommand{\inN}{\ensuremath{\in\mathbb{N}}\xspace}
\newcommand{\inQ}{\ensuremath{\in\mathbb{Q}}\xspace}
\newcommand{\inR}{\ensuremath{\in\mathbb{R}}\xspace}
\newcommand{\inC}{\ensuremath{\in\mathbb{C}}\xspace}
\newcommand{\OO}{\ensuremath{\mathcal{O}}\xspace}
\newcommand{\rarr}{\rightarrow}
\newcommand{\Rarr}{\Rightarrow}
\newcommand{\xrarr}[1]{\xrightarrow{#1}}
\newcommand{\larr}{\leftarrow}
\newcommand{\lrarr}{\leftrightarrows}
\newcommand{\rlarr}{\rightleftarrows}
\newcommand{\rrarr}{\rightrightarrows}
\newcommand{\al}{\alpha}
\newcommand{\bt}{\beta}
\newcommand{\ld}{\lambda}
\newcommand{\om}{\omega}
\newcommand{\Kd}[1]{\ensuremath{\Omega^{#1}}}
\newcommand{\KKd}{\ensuremath{\Omega^2}}
\newcommand{\vd}{\partial}
\newcommand{\PC}{\ensuremath{\mathbb{P}_1(\mathbb{C})}}
\newcommand{\PPC}{\ensuremath{\mathbb{P}_2(\mathbb{C})}}
\newcommand{\derz}{\ensuremath{\frac{\partial}{\partial z}}}
\newcommand{\derw}{\ensuremath{\frac{\partial}{\partial w}}}
\newcommand{\mb}[1]{\ensuremath{\mathbb{#1}}}
\newcommand{\mf}[1]{\ensuremath{\mathfrak{#1}}}
\newcommand{\mc}[1]{\ensuremath{\mathcal{#1}}}
\newcommand{\id}{\ensuremath{\mbox{id}}}
\newcommand{\dd}{\ensuremath{\delta}}
\newcommand{\bu}{\bullet}
\newcommand{\ot}{\otimes}
\newcommand{\boxt}{\boxtimes}
\newcommand{\op}{\oplus}
\newcommand{\mt}{\times}
\newcommand{\Gm}{\mathbb{G}_m}
\newcommand{\Ext}{\ensuremath{\mathrm{Ext}}}
\newcommand{\Tor}{\ensuremath{\mathrm{Tor}}}

\newcommand{\Hom}{\ensuremath{\mathrm{Hom}}}

\newcommand{\kn}[1]{\mathrm K(n)^*(#1)}
\newcommand{\ckn}[1]{\mathrm{CK}(n)^*(#1)}
\newcommand{\grckn}[1]{\mathrm{gr}_\tau^{*}\,\mathrm{CK}(n)^{*}(#1)}
\newcommand{\so}{\mathrm{SO}_m}
\newcommand{\pt}{\mathrm{pt}}
\newcommand{\tr}{\mathrm{tr}}
\newcommand{\sic}{\mathrm{sc}}
\newcommand{\ad}{\mathrm{ad}}
\newcommand{\sr}{\mathrm{sr}}
\newcommand{\St}{\mathrm{St}}
\newcommand{\SK}{\mathrm{SK}}
\newcommand{\SL}{\mathrm{SL}}
\newcommand{\Cp}{\mathrm{Cp}}
\newcommand{\Sp}{\mathrm{Sp}}
\newcommand{\Ep}{\mathrm{Ep}}
\newcommand{\Spin}{\mathrm{Spin}}

\newcommand{\A}{\mathsf{A}}
\newcommand{\C}{\mathsf{C}}

\def\GF#1{{\mathbb F}_{\!#1}}

\newcommand{\Pic}{\mathrm{Pic}}



\newtheorem{lm}{Lemma}[section]
\newtheorem{lm*}{Lemma}
\newtheorem*{tm*}{Theorem}
\newtheorem*{tms*}{Satz}
\newtheorem{tm}[lm]{Theorem}
\newtheorem{prop}[lm]{Proposition}
\newtheorem*{prop*}{Proposition}
\newtheorem{prob}{Problem}
\newtheorem{cl}[lm]{Corollary}
\newtheorem*{cor*}{Corollary}
\newtheorem{conj}{Conjecture}
\theoremstyle{remark}
\newtheorem*{rk*}{Remark}
\newtheorem*{rm*}{Remark}
\newtheorem{rk}[lm]{Remark}
\newtheorem*{xm}{Example}

\theoremstyle{definition}
\newtheorem{df}{Definition}
\newtheorem*{nt}{Notation}
\newtheorem{Def}[lm]{Definition}
\newtheorem*{Def-intro}{Definition}
\newtheorem{Rk}[lm]{Remark}
\newtheorem{Ex}[lm]{Example}

\newtheorem{Qu}[lm]{Question}

\theoremstyle{plain}
\newtheorem{Th}[lm]{Theorem}
\newtheorem*{Th*}{Theorem}
\newtheorem*{Th-intro}{Theorem}
\newtheorem{Prop}[lm]{Proposition}
\newtheorem*{Prop*}{Proposition}
\newtheorem{Cor}[lm]{Corollary}
\newtheorem{Lm}[lm]{Lemma}
\newtheorem*{Conj}{Conjecture}
\newtheorem*{BigTh}{Classification of Operations Theorem  (COT)}
\newtheorem*{BigTh-add}{Algebraic Classification of Additive Operations Theorem  (CAOT)}

\newtheorem{maintheorem}{Theorem}
\renewcommand{\themaintheorem}{\Alph{maintheorem}}

\newtheorem*{ThA}{Theorem A}
\newtheorem*{ThB}{Theorem B}
\newtheorem*{ThSh}{Shafarevich's Theorem}

\newcommand\rA{\mathsf A}
\newcommand\rB{\mathsf B}
\newcommand\rC{\mathsf C}
\newcommand\rD{\mathsf D}
\newcommand\rE{\mathsf E}
\newcommand\rF{\mathsf F}
\newcommand\rG{\mathsf G}

\newcommand{\ep}{\epsilon}
\newcommand{\be}{\beta}
\newcommand{\ga}{\gamma}
\newcommand{\de}{\delta}
\newcommand{\la}{\lambda}
\newcommand{\vp}{\varphi}
\newcommand{\st}{\sigma}
\newcommand{\eps}{\varepsilon}

\newcommand{\lra}{\longrightarrow}
\newcommand{\idd}{\mathop{\rm{id}}\nolimits}
\newcommand{\Real}{{\mathbb R}}
\newcommand{\Co}{{\mathbb C}}
\newcommand{\komp}{{\mathbb C}}
\newcommand{\Int}{{\mathbb Z}}
\newcommand{\Nat}{{\mathbb N}}
\newcommand{\Rat}{{\mathbb Q}}

\def\mq{{\mathfrak q}}

\def\bark{\bar k}




\title[Linearization of finite subgroups of Cremona groups]{
Linearization of finite subgroups of Cremona groups\\ over non-closed fields
}

\author{Boris Kunyavski\u\i }
\address{%
Department of Mathematics \\
Bar-Ilan University \\
5290002 Ramat Gan, Israel
}
\email{Boris.Kunyavskii@biu.ac.il, kunyav@macs.biu.ac.il}



\thanks{Author's research was supported by the ISF grant 1994/20.
Part of this work was accomplished when he was visiting the IHES (Bures-sur-Yvette), Sorbonne Universit\'e (Paris),
and the MPIM (Bonn). Support of these institutions is gratefully acknowledged.}

\begin{abstract}
We study linearizability properties of finite subgroups of the Cremona group $\Cr_n(k)$
in the case where $k$ is a global field, with the focus on the local-global principle.
For every global field $k$ of characteristic different from 2
and every $n \ge 3$ we give an example of a birational involution of $\mathbb P^n_k$
(=an element $g$ of order $2$ in $\Cr_n(k)$) such that
$g$ is not $k$-linearizable but $g$ is $k_v$-linearizable in $\Cr_n(k_v)$ for all places $v$ of $k$.
The main tool is a new birational invariant generalizing those introduced by Manin and Voskresenski\u\i\
in the arithmetic case and by Bogomolov--Prokhorov in the geometric case. We also apply it to
the study of birational involutions in real plane.
\end{abstract}

\maketitle

\section{Introduction}

The main object of our interest in the present paper is the Cremona group $\Cr_n(k)$.
It is defined as the group of birational automorphisms of the projective space $\mathbb P^n_{k}$.
More precisely, our focus is on {\it linearizability} properties of subgroups $G$ of $\Cr_n (k)$.


\begin{Def}
We say that $G\subset \Cr_n (k)$ is {\emph{projectively $k$-linearizable}} if it is conjugate in $\Cr _n(k)$ to
a subgroup $G'$ of $\PGL_{n+1}(k)$. If one can choose $G'$ so that the faithful projective representation
$G'\hookrightarrow \PGL_{n+1}(k)$ can be lifted to a linear representation $G'\to \GL(n,k)$, we say that
$G$ is {\emph{$k$-linearizable}}.


\end{Def}

Since $\PGL_{n+1}(k)$ is the group of automorphisms of $\mathbb P^n_{k}$, one can rephrase this saying that
$\mathbb P^n_{k}$ admits a faithful regular action of $G$. More precisely, given $G\subset\Cr_n (k)$, we say
that $G$ is {\it $k$-regularizable} if there exist a smooth projective $k$-variety $X$ and a $k$-birational map
$\varphi \colon X\dashrightarrow \mathbb P^n_{k}$
such that $\varphi^{-1}\circ G\circ \varphi\subset \Aut (X)$.
Thus, if one can take $X=\mathbb P^n_{k}$, then
$G$ is projectively $k$-linearizable.

Not all subgroups $G\subset\Cr_n (k)$ are $k$-regularizable, see
various examples in \cite{Bru}, \cite{Ch} (attributed there to Dolgachev), \cite{Ca}. Moreover,
the set $R$ of regularizable subgroups is meagre enough: recently Lin and Shinder \cite{LS} showed that
$R$ may not (and normally does not) generate $\Cr_n (k)$, thus refuting a long-standing conjecture.

However, for {\it finite} subgroups $G\subset \Cr_n (k)$ the situation is entirely different: {\it any} such $G$
is regularizable, at least if $\Char (k)=0$, see, e.g. \cite{Ch}.

In the present paper, we assume throughout that $G$ is {\it finite}. Our general goal is to study the behaviour of
the linearizability of $G$ under base change, when $k$ is a subfield of $K$, and we compare $G\subset \Cr_n (k)$
to its image $G_K$ with respect to the natural inclusion $\Cr_n (k)\subset \Cr_n(K)$. This topic
got certain attention in the classical complexification-realification context, when $k=\mathbb R$, $K=\mathbb C$,
see \cite{Ya1}, \cite{Ya2}, \cite{CMYZ}. We shall also touch this situation though our focus is on another, arithmetical
set-up. Namely, our main interest is in the validity of the following local-to-global principle.

\begin{Qu}  \label{q:main}
Let $k$ be a global field, let $G\subset \Cr_n (k)$ be a finite subgroup,
and for each place $v$ of $k$ let $G_v$ denote the image of $G$ in $\Cr_n (k_v)$
with respect to the natural inclusion $k\subset k_v$. Suppose that $G_v$ is
(projectively) $k_v$-linearizable for every $v$. Is it true that $G$ is (projectively) $k$-linearizable?
\end{Qu}

Our main result answers this question in the negative for any global field $k$ of characteristic different from 2
and any $n\ge 3$. It turns out that counter-examples can be found among subgroups of order $2$.

\begin{ThA}  \label{thA} (cf. Theorem \ref{th:main})
For any $n\ge 3$ and any global field $k$ of characteristic different from 2
there exists a birational involution
of $\mathbb P^n_{k}$ $($= an element $g$ of order $2$ in $\Cr_n (k))$ such that

$\bullet$ $g$ is $k_v$-linearizable for all $v$;

$\bullet$ $g$ is not $k$-linearizable.
\end{ThA}

The proof contains several ingredients of different nature. The first is a general
statement providing a necessary condition for the linearizability. It is formulated in the
set-up of equivariant birational geometry over an arbitrary field, and I believe
it is interesting in its own right.

Let $k$ be a field, and
let $X$, $Y$ be smooth, projective, integral $k$-varieties equipped with
a faithful $k$-regular action of a finite group (=constant $k$-group scheme) $G$ (for short,
{\it {$G$-varieties}}).  We say that
 $G$-varieties $X$, $Y$ are
{\it {$G$-birational}} if there exists a $k$-birational map
$\varphi\colon X \dashrightarrow Y$ {\it equivariant} with respect to the action of $G$. We say that
a smooth, projective, integral $k$-variety
$X$ is {\it {$k$-linearizable}} if it is $G$-birational to $Y=\mathbb P(V)$ where the $G$-action on $Y$
is induced by a linear action of $G$ on a vector $k$-space $V$.

Let $\bar k$ denote a separable closure of $k$.
Denote $\overline X=X\times_{k}\bark$.
Let $\Gamma=\Gal(k)=\Gal(\bark/k)$ be the absolute Galois group of $k$.
The actions of $G$ and $\Gamma$ on $\overline X$ commute, so that we can define the action
of $W:=G\times\Gamma$ on $\overline X$ and extend it to the geometric Picard group
$\Pic (\overline X)$. Denote by $N$ the $W$-module $\Pic (\overline X)$.

Suppose that $N$, viewed as a $\mathbb Z$-module, is a free module of finite rank (this holds, say,
if $\overline X$ is a rational variety). We say that $N$ is a {\it permutation} module if it has a free $\mathbb Z$-basis
globally respected by $W$. We say that $N$ is {\it stably permutation} if there exist permutation modules $S$, $S'$ such that
$N\oplus S\cong S'$. We say that $W$-modules $N$ and $N'$ are {\it{similar}} if there exist permutation modules $S$, $S'$ such that
$N\oplus S\cong N'\oplus S'$.(More properties of a similar flavour will be introduced in Section \ref{sec:prelim}.)

The following theorem, the proof of which (as well as of its stronger version, Theorem \ref{th:crit}) 
was communicated to me by Jean-Louis Colliot-Th\'el\`ene, is crucial for the proof of Theorem A.

\begin{ThB} \label{thB}  (cf. Theorem \ref{th:crit})
With the above notation,  suppose that the $G$-variety $X$ is  $k$-linearizable.
Then $N=\Pic (\overline X)$
is a stably permutation $W$-module.
 \end{ThB}

The second ingredient in the proof of Theorem A is an explicit construction
of a birational involution of a certain 3-dimensional algebraic $k$-torus $T$.
It is a modification of \cite[Example~3.7]{Ku2}. This purely algebraic result
allows us to prove that the geometric Picard module $\Pic (\overline X)$
of a smooth compactification $X$ of $T$ is {\it not}
stably permutation {for the $W$-action} and
thus deduce from Theorem B that the corresponding birational involution is {\it not}
$k$-linearizable.

To finish the proof of Theorem A and prove that this involution is everywhere locally
linearizable, we use (in the same way as in \cite{Ku2})
the existence of a biquadratic extension $K/k$ all decomposition groups of which are cyclic.
This is actually the only arithmetical ingredient of the proof.


\begin{Rk} \label{rem:stable}
Below we shall prove slightly more general versions of Theorems A and B. Namely, we shall use
the {\it stable} linearizability condition,
 see details in Section \ref{sec:proof}. Recall that there are several subtle
examples showing the existence of a gap between stable rationality and
rationality in the arithmetic case \cite{BCTSSD}, and  between
stable linearizability and linearizability in the geometric case
\cite{RY}, \cite{LPR}, \cite{HT}, \cite{BvBTH}, \cite{BvBT}, \cite{CTZ}, \cite{PSY}.
\end{Rk}

\section{Preliminaries on tori} \label{sec:prelim}
In this section, we provide a brief account of some basic facts on algebraic tori necessary
for the proof of Theorem A. Details can be found in \cite{CTS1}, \cite{Vo}.

$\bullet$ A $k$-torus $T$ is an algebraic $k$-group such that $T\times_kk^{\textrm{alg}}\cong \mathbb G_{\textrm{m},k^{\textrm{alg}}}^d$,
where $k^{\textrm{alg}}$ is an algebraic closure of $k$.

$\bullet$ Every torus $T$ splits over some finite separable extension of $k$ (Ono, Borel, Springer, Tate, Tits),
see \cite{Yu} for an overview of different proofs. Hence in the definition above one can replace $k^{\textrm{alg}}$
with $\bar k$. There is a unique smallest finite Galois subextension $L/k$ inside $\bar k/k$ which splits $T$.
We call $L$ the minimal splitting field of $T$.

$\bullet$
Let $L$ be the minimal splitting field of $T$, $\Pi=\Gal (L/k)$,
then $M=\hat T=\Hom (T,\mathbb G_{\mathrm m})$ viewed as a $\Pi$-module
is called the character module of $T$.

$\bullet$
Many rationality properties of $T$ can be expressed in terms of $M$.

We say that a $\Pi$-module $N$ is
\begin{enumerate}
\item [(i)] {\it permutation} if it has a $\mathbb Z$-base globally respected by $\Pi$;

\item
[(ii)] {\it stably permutation} if $N\oplus S\cong S'$ for some permutation modules $S$, $S'$;
\item
[(iii)] {\it invertible} if $N$ is a direct summand of a permutation module;
\item
[(iv)] $H^1$-{\it trivial} (aka coflasque) if $H^1(\Pi',N)=0$ for all subgroups $\Pi'\le \Pi$;
\item
[(v)] {\it flasque} if the dual module $N^{\circ}:=\Hom (N,\mathbb Z)$ is $H^1$-trivial.
\end{enumerate}
We have irreversible implications $$\mathrm{(i)\Rightarrow(ii)\Rightarrow(iii)\Rightarrow(iv)\cap (v)\Rightarrow(iv)({\textrm{ or}} (v))}.$$

$\bullet$ Any module $M$ can be embedded into a short exact sequence
$$
0\to M\to S\to F\to 0,
$$
where $S$ is permutation and $F$ is flasque; such a sequence is called a flasque resolution of $M$.

The module $F$ can be constructed geometrically: embed $T$ into a smooth projective variety $X$
as an open subset (this is possible even in positive characteristic \cite{CTHS}),
then $F=\Pic (X\times_k\bar k)$.


$\bullet$ We say that a  $k$-torus $T$ is
\begin{enumerate}
\item[(i)] $k$-{\it rational} if $T$ is birationally $k$-equivalent to $\mathbb A^d$;
\item[(ii)] {\it stably} $k$-{\it rational} if $T\times \mathbb A^m$ is $k$-rational for some $m\ge 0$;
\item[(iii)]  {\it retract rational}  if $T\times T'$ is $k$-rational for some torus $T'$.
\end{enumerate}

We then have the implications $\mathrm{(i)\Rightarrow(ii)\Rightarrow(iii)}$; the right one is irreversible
whereas the reversibility of the left one is a notoriously difficult long-standing problem.


\begin{Rk}
In the original definition \cite{Sa}, the property of retract rationality is different. For tori,
it implies the property of item (iii), as proved in \cite[Proposition~7.4]{CTS3}.
\end{Rk}

$\bullet$ We have the following relations:
$T$ is stably rational (resp. retract rational) if and only if the module $F$ in a flasque
resolution of $M=\hat T$ is stably permutation (resp. invertible).

$\bullet$ We shall use the notation $R_{K/k}$ for Weil's restriction of scalars from $K$ to $k$.

\section{Proofs of main results} \label{sec:proof}

We start with reformulations of Theorems A and B in a slightly more general way,
as mentioned in the introduction.  We maintain the notation introduced there.

\begin{Def} \label{def:stable}
We say that
smooth, projective, integral
$G$-varieties $X$, $Y$ defined over a field $k$ are {\emph{stably $G$-birational}} if there exist $m,n$ such that
$X\times \mathbb P^n_{k}$ and $Y\times \mathbb P^m_{k}$, where $G$ acts trivially on
the projective spaces, are $G$-birational.

We say that $X$ is {\emph{stably projectively $k$-linearizable}}
if it is stably $G$-birational to a projective space $\mathbb P^m_k$ with trivial $G$-action.

We say that $X$ is {\emph{stably $k$-linearizable}}
if it is stably $G$-birational to a projective space $\mathbb P(V)$ equipped with the action
induced by a linear action of $G$ on a vector $k$-space $V$.
\end{Def}

\begin{Rk}
Note that there are different definitions of (stable) linearizability, see \cite{Pr}, \cite{CTZ}.
\end{Rk}






The linearizability properties introduced in Definition \ref{def:stable} fit into
the following diagram of implications:

$$
\begin{array}
[c]{ccc}
k{\rm{-linearizable }} & \Rightarrow & {\rm{stably}} \, k{\rm{-linearizable }}\\
\Downarrow &  & \Downarrow  \\
{\rm{projectively}} \, k{\rm{-linearizable }} & \Rightarrow & {\rm{stably}} \, {\rm {projectively}} \, k{\rm{-linearizable }}
\end{array}
$$



\bigskip

Before proceeding to the precise statements and proofs, one has to ask a natural question
regarding the behaviour of the linearization properties with respect to taking direct products.

\begin{Qu} \label{q:prod}
Let $X$, $Y$ be smooth projective $k$-varieties equipped with a faithful action of a finite group $G$.

(i) Suppose that $X$ and $Y$ are $k$-linearizable. Is it true that so is $X\times_kY$?

(ii) Suppose that $X$ and $Y$ are projectively $k$-linearizable. Is it true that so is $X\times_kY$?

\end{Qu}

It turns out that the answers to (i) and (ii) are different. The following remarks are due to Yuri Tschinkel.

\begin{Rk} \label{rm:prodprojlin}
Question \ref{q:prod}(i) is answered in the affirmative. The proof is based on a projective version of
the no-name lemma (see \cite{Kr}), which can be easily deduced from \cite[Lemma~3.5]{CTS4}. 
\end{Rk}

\begin{lm} \label{lem:noname}
Let $X$ be a projective $G$-variety. Assume that the $G$-action is generically free.
Let $\mathcal E \to X$ be a $G$-vector bundle of positive rank $n$. Then we have a birational isomorphism
\begin{equation} \label{eq:noname}
\mathbb P(\mathcal E)\sim_GX\times \mathbb P^{n-1},
\end{equation}
where on the right-hand side the action of $G$ on the second factor is trivial.
\end{lm}

With Lemma \ref{lem:noname} at hand, for $X=\mathbb P(V)$ and $Y=\mathbb P(V')$ we view $\mathbb P(V)\times \mathbb P(V')$
as the projectivization of a trivial $G$-vector bundle over $\mathbb P(V)$, 
and (\ref{eq:noname}) implies that
$\mathbb P(V)\times \mathbb P(V') \sim_G \mathbb P(V)\times \mathbb P^{\dim(V')-1}$, with trivial action on the second factor.
Then the same argument puts the trivial part into the representation. (This argument is implicit in \cite[Lemma~1]{Ka}.)

\begin{Rk} \label{rm:prodlin}
Question \ref{q:prod}(ii) is in general answered in the negative, even when $k$ is algebraically closed,  because projective linear actions have cyclic
Amitsur group ${\mathrm{Am}}^2(X,G)$, and a product of two projective linear actions with nontrivial Amitsur group will give a noncyclic group.
See \cite{KrT} for details on the Amitsur invariant.
\end{Rk}

We now go over to the precise statements of the main results.

Theorem B follows from the following stronger version.

\begin{Th} \label{th:crit} (Colliot-Th\'el\`ene)
Suppose that smooth, projective, geometrically integral
 $G$-varieties $X$, $X'$
 over the field $k$
are stably projectively $G$-birational. Then the $W$-modules $N=\Pic (\overline X)$
and $N'=\Pic (\overline X')$,
where $\overline X=X\times_k\bar k$ and $\bark$ is a separable closure of $k$,
are similar, in particular $H^1(U,N)=H^1(U,N')$ for any subgroup $U\le W$.

Thus if $X$ is stably projectively $k$-linearizable, then $N$ is a stably permutation $W$-module,  in particular,
it is flasque and coflasque.
 \end{Th}

\begin{Rk} \label{rem:equi}
The earlier versions of Theorem \ref{th:crit} were established in two extreme situations: the arithmetic case, when
$k\ne\bark$ and $G=\{1\}$, was settled by  Manin \cite{Ma} for  surfaces and by
Voskresenskii \cite[Chapter~2, Section 4.4]{Vo} in general, under resolution of singularities,
and the geometric case,
when $k = \bark$ and $G\ne\{1\}$, is due to Bogomolov and Prokhorov \cite{BP}
{under resolution of singularities}.
\end{Rk}




We postpone the proof of Theorem \ref{th:crit} until Section \ref{sec:crit}.

\medskip

Before proving Theorem A, we also make its statement a little more general,
as for Theorem B.

First note that for $m>n$ we have a natural inclusion of Cremona groups
$\Cr_n(k)\subset\Cr_m(k)$ induced by adjoining variables. We say that
subgroups $G,G'\subset\Cr_n(k)$ are {\emph{stably conjugate}} if they
are conjugate in $\Cr_m(k)$ for some $m\ge n$. We say that $G\subset\Cr_n(k)$
is {\emph{stably (projectively) $k$-linearizable}} if it is stably conjugate to some
(projectively) $k$-linearizable subgroup.

\begin{Rk} \label{rem:reg}
Let $k$
 be an arbitrary field, $G$ a finite group,
$X$ a
smooth, projective, geometrically integral,
{\emph{$k$-rational}} $G$-variety of dimension $n$.
Assume that the action of $G$ on $X$ is faithful.
Then, as in the Introduction,
using a $k$-birational map $\varphi\colon X\dashrightarrow \mathbb P^n_{k}$,
one can embed $G$ into $\Cr_n(k)$ by mapping $g\in G$ to $\varphi\circ g\circ \varphi^{-1}$.

With this correspondence, $G\subset \Cr_n(k)$
is (stably) (projectively) $k$-linearizable if and only if
$X$ is (stably) (projectively) $G$-linearizable.
\end{Rk}

Theorem A follows from the following stronger version.

\begin{Th}  \label{th:main}
For any $n\ge 3$ and any global field $k$
of characteristic different from 2
there exists a birational involution
of $\mathbb P^n_{k}$ $($= an element $g$ of order $2$ in $\Cr_n (k))$ such that

$\bullet$ $g$ is $k_v$-linearizable for all $v$;

$\bullet$ $g$ is not stably projectively $k$-linearizable.
\end{Th}

\begin{proof}
\noindent{\emph{Step 1.}} We start with the case $n=3$.

Let $L=k(\sqrt{a}, \sqrt{b})$ be a biquadratic extension. Embed $\mathbb G_{\mathrm m,k}$
into $R_{L/k}(\mathbb G_{\mathrm{m},L})$
and consider the quotient
$T=R_{L/k}(\mathbb G_{\mathrm{m},L})/\mathbb G_{\mathrm m,k}$.
The torus $T$ is $k$-rational, being an open subset of
$$R_{L/k}(\mathbb A^1_{L})/\mathbb G_{\mathrm m,k}=\mathbb A^4_{k}/\mathbb G_{\mathrm m,k}=\mathbb P^3_{k}.$$
The character module $\hat T$ is isomorphic to $\ker\left[\mathbb Z[\Pi]\to\mathbb Z\right]$, the augmentation ideal
of $\mathbb Z[\Pi]$, see \cite[Section~4.8]{Vo}.

It is convenient to use a matrix representation of $\hat T$. Fix a pair $\rho, \sigma$ of generators of $\Pi$, a basis of $\mathbb Z^3$,
and an isomorphism
of modules $\hat T$ and $\mathbb Z^3$. Consider the
subgroup of $\GL(3,\mathbb Z)$ corresponding to the $\Pi$-action on $\hat T$. It is conjugate
to
\begin{equation} \label{eq:I}
\left<
         \left(\begin{matrix} 0 & 0 & 1 \\ -1 & -1 & -1 \\ 1 & 0  & 0\end{matrix} \right),
 \left(\begin{matrix} -1 & -1 & -1 \\ 0 & 0 & 1 \\ 0 & 1 &  0 \end{matrix} \right)
\right>.
\end{equation}


\noindent{\emph{Step 2.}} Define a birational involution $g$ of $T$ sending $t$ to $t^{-1}$.
More formally, we define the action of $g$ on $\hat T$ (written additively)
as sign changing, so that the corresponding matrix is $-I_3$. Denote $G=\left<g\right>$.

Embed $T$ into a smooth projective model $X$ as an open subset.
Recall (see Section \ref{sec:prelim}) that such an $X$ exists  even when $\Char (k) > 0$.
Moreover, note that such an embedding can be made {\emph {$G$-equivariant}}. Indeed,
in the set-up of Theorem \ref{th:crit}, replace $\Gamma=\Gal(\bar k/k)$ with
its finite quotient $\Pi=\Gal (L/k)$ and  define $W=G\times\Pi$. Since the
$G$-action on the algebraic variety $T$ comes from its action by automorphisms of the lattice $\hat T$,
we can use \cite[Th.~1]{CTHS} to obtain a $W$-equivariant, smooth, projective fan from which we
obtain the needed $G$-variety $X$.

\medskip

\noindent{\emph{Step 3.}} Consider now the $W$-module $\Pic(\overline X)$.

Looking at the finite subgroup of $\GL(3,\mathbb Z)$ corresponding to the $W$-action on $\hat T$, we see from \eqref{eq:I}
and the definition of the action of $g$ that it is conjugate to
 \begin{equation} \label{eq:II}
\left<\left(\begin{matrix} 0 & 0 & 1 \\ -1 & -1 & -1 \\ 1 & 0  & 0\end{matrix} \right),
      \left(\begin{matrix} -1 & -1 & -1 \\ 0 & 0 & 1 \\ 0 & 1 &  0 \end{matrix} \right) ,
         \left(\begin{matrix} -1 & 0 & 0 \\ 0 & -1 & 0 \\ 0 & 0  & -1\end{matrix} \right)
\right>.
\end{equation}

Looking at the list of finite subgroups of $\GL(3,\mathbb Z)$ corresponding to
 tori that are not stably rational \cite[Theorem~1]{Ku1}, we see that subgroup \eqref{eq:II} appears
on that list where it is denoted by $W_2$. According to \cite[Section~4]{Ku1},
the flasque $W$-module $F=\Pic (\overline X)$
is not invertible (and hence not stably permutation).

See \cite[Section~7]{HY} for a computer verification of this result, see \cite[Section~5]{CTZ}
for an explicit geometric description of the toric variety $X$.

Note that the $W$-module $\Pic (\overline X)$ is flasque and coflasque,
see \cite[Section~3]{Ku2}, \cite{HY};
this allowed the use of the torus $T'$ in \cite[Example~3.7]{Ku2} for constructing an example of
a toric variety violating a local-to-global principle for rationality in the absence of the Brauer obstruction.

According to Theorem \ref{th:crit}, $G$ is not stably projectively $k$-linearizable.

\medskip

\noindent{\emph{Step 4.}}
Let us now choose $L$ 
such that all decomposition groups $\Pi_v$ are cyclic
(hence of order 1 or 2). Such a choice is always possible, see \cite[Corollary~9.2.3]{NSW}.
Thus each $k_v$-torus $T_v=T\times_{k}k_v$ splits by a quadratic extension $K:=k_v(\sqrt{d})$.
Then $T_v$ is a direct product of tori of dimension 1 or 2 \cite[Chapter~2, \S5, Theorem~4]{Vo}. 
By Remark \ref{rm:prodprojlin}, it is enough to show that the action of $G$ on the 1- and 2-dimensional direct factors of $T_v$
is linearizable.

(Note that even in the case $n=2$, the linearizability of a birational involution is not automatic, even if it is linearizable
after a quadratic base change, see, e.g. \cite[Section~8]{CMYZ}.)

For the one-dimensional split torus $\mathbb G_{\rm{m},k_v}$ the needed property is immediate: it equivariantly embeds into $\mathbb P^1$ where $g$ acts by swapping the coordinates,
and the projective representation $G\to\PGL_2(k_v)$ corresponding to this action lifts to a linear representation.

Let now $T_v$ be the norm one torus $R_{K/k_v}^1{\mathbb G}_{\textrm{m,K}}$ with the $G$-action given by $g(t)=t^{-1}$. Let $\gamma$ denote
the generator of $\Gal (K/k_v)$. Then $T_v$ can
be represented by the equation $N(x)=x\bar x=1$, where $\bar x=\gamma (x)$, and hence the action of $g$ coincides with the action of $\gamma$.

One can view $T_v$ as a smooth affine conic $C\mathbf \subset  \mathbb A^2$ given by $x^2-dy^2=1$, it has a $G$-invariant rational point $P=(1,0)$.
The smooth equivariant compactification $\overline C \subset \mathbf P^2$ given by $X^2-dY^2=Z^2$ has a $G$-invariant rational point $\overline P=(1:0:1)$.
We have an isomorphism $f\colon \overline C \overset{\cong}{\to}\mathbf P^1$, with the $G$-action induced by the $G$-action on $\overline C$ and
$G$-invariant rational point $f(\overline P)$. This regular action gives rise to a projective representation $G\to \PGL(2,k_v)$. The presence of a
$G$-invariant rational point implies that the image of $g$ is conjugate to $\left(\begin{matrix} 1 & 0 \\ 0 & -1\end{matrix} \right)$ (modulo scalar
matrices). Hence this representation can be lifted to a linear representation $G\to \GL(2,k_v)$.
\begin{Rk}
There are projectively linearizable actions of the group of order 2 that are not
linearizable (even one-dimensional). I thank Yuval Ginosar for indicating the following example.
Let $k$ be a field of characteristic different from 2 that is
not quadratically closed, let $a\in k^*\setminus k^{*2}$,  and
let $G=\mathbb Z/2\mathbb Z=\left<g\right>$ act on $\mathbb G_{\textrm{m},k}$ by $g(t)=a/t$. This action
gives rise to the projective representation $\varphi\colon G\to \PGL_2(k)$ with $\varphi(g)$ conjugate to
$\left(\begin{matrix} 0 & a \\ 1 & 0\end{matrix} \right)$ (modulo scalar matrices), which cannot be lifted
to a linear representation. (Such a phenomenon cannot occur in the traditional setting where
$k$ is algebraically closed of characteristic zero because then the Schur multiplier $H^2(G,k^*)$ is zero for
any cyclic group $G$.)
\end{Rk}

To finish the proof of the case $n=3$ of the theorem, we have to consider two-dimensional tori appearing in a direct decomposition of $T_v$.
Such a torus must be isomorphic either to a product of two one-dimensional tori (when Remark \ref{rm:prodprojlin} is applicable),
or to a torus of the form $S=R_{L/k_v}\mathbb G_{\textrm{m},L}$,
where $L/k_v$ is a quadratic field extension.

We first embed $S$ into a $k_v$-quadric $Q=R_{L/k_v}\mathbb P^1_L$ as an open subset, with the $G$-action on $Q$
induced by the $G$-action on $\mathbb G_{\textrm{m},k_v}$ (recall that Weil's restriction functor is compatible with
group actions, see \cite[Section~5]{Bri2}). More explicitly, the open set $x^2-dy^2 \neq 0$ in $\mathbb A^2$
can be identified with the open set $zt\neq 0$ in the quadric $x^2-dy^2+zt=0$.

To show that the $G$-action on $Q$ is $k_v$-linearizable, we use the following
lemma. Its proof (for the case $k=\mathbb C)$ and the subsequent remark were communicated to me by Ivan Cheltsov.
The case of general $k$ requires certain adjustments (suggested by Jean-Louis Colliot-Th\'el\`ene). 

\begin{lm} \label{lem:Cheltsov}
Let $G$ be a finite group, let $n\ge 3$, and let $Q\subset \mathbb P_k^n$ be a quadric hypersurface equipped with
a faithful action of $G$. Assume that the order of $G$ is prime to $n$. 
Suppose that $Q$ has a $k$-point fixed by $G$. Then the action of $G$ is $k$-linearizable.
\end{lm}

\begin{proof}
Projecting from the $G$-fixed $k$-point, we get a $k$-birational map $Q\dashrightarrow \mathbb P_k^{n-1}$.
This gives rise to a projective $k$-representation $G\to\PGL_n$. 
Consider the degree $n$ covering $\SL_n\to \PGL_n$. 
The long sequence of cohomology yields an exact sequence of groups 
\begin{equation} \label{covering} 
1 \to \mu_{n}(k) \to \SL_{n}(k) \to \PGL_{n}(k) \to k^*/k^{*n}
\end{equation}
where $\mu_n$ is the $k$-group scheme of $n^{\textrm{th}}$ roots of unity. Sequence \eqref{covering} 
induces short exact sequences 
$$
1 \to \mu_{n}(k) \to \SL_{n}(k) \to \SL_{n}(k)/\mu_{n}(k) \to 1
$$
and 
$$
1 \to \SL_{n}(k)/\mu_{n}(k)  \to \PGL_{n}(k) \to  k^*/k^{*n}. 
$$
Since we have an embedding $G \hookrightarrow \PGL_{n}(k)$ and 
the order of $G$ is prime to $n$, we conclude that $G$ embeds into 
$\SL_{n}(k)/\mu_{n}(k)$. By pull-back, we obtain a group extension 
\begin{equation} \label{extension}
1 \to \mu_{n}(k) \to G' \to G  \to 1
\end{equation}
where $G'$  is a subgroup of $\SL_{n}(k)$ and 
$\mu_{n}(k)$ is a central abelian
subgroup of $G'$ of order prime to the order of $G$.
Such extensions are classified by the group 
$H^2(G,\mu_{n}(k))$, which is killed by the order of $G$ and 
the order of $\mu_{n}(k)$ and is therefore zero. 

Hence extension \eqref{extension} admits a section, a group homomorphism $G \to G'$, 
whence a required embedding $G\hookrightarrow \SL_{n}(k)$.
\end{proof}

\begin{Rk}
The statement of Lemma \ref{lem:Cheltsov} (for $k=\mathbb C$) was generalized by Alexander Duncan, see \cite[Lemma~2.4]{ACKM}.
\end{Rk}

Thus the case $n=3$ is settled.

\medskip

\noindent{\emph{Step 5.}} To treat the case of $G\subset \Cr_{l}(k)$, $l>3$, one can take the involution acting on the first three
variables as $g$ considered above and trivially on all the remaining variables. The resulting involution on
$T\times\mathbb G_{\textrm{m}}^{l-3}$ extends to $X_l:=X\times \mathbb P^{l-3}$ with the three-dimensional $X$
as above and the trivial action of $G$ on the projective space. Therefore the $W$-module $\Pic (\overline X_l)$
is isomorphic to $\Pic (\overline X) \oplus \mathbb Z^{l-3}$ and is thus not stably permutation, as $\Pic (\overline X)$,
which proves that $G$ is not stably projectively $k$-linearizable. The local linearizability is proven in exactly the same way as
in the case $n=3$.

This finishes the proof of Theorem \ref{th:main}.
\end{proof}




\section{Proof of Theorem \ref{th:crit}} \label{sec:crit}
As said in the introduction, the proof presented below was communicated to me by
Jean-Louis Colliot-Th\'el\`ene.

Let $G$ be a finite group. Let $k$ be a field (no restriction on the field).
In what follows, a $G$-action on a $k$-variety $X$ is always  assumed to be
an action over $k$.


We start with a couple of lemmas.

\begin{lm} \label{Lemma 1}
Let $G$ be a finite group. Let $X/k$ be a quasi-projective variety
equipped with a $G$-action. Then $X$ can be covered by affine open sets which are
$G$-stable.
\end{lm}

\begin{proof} Let $m \in X$ be a closed point. Since $X$ is quasiprojective, the finitely many points $g.m$
are contained in an affine open set $U$.
The intersection of the affine open sets $gU$ is
a $G$-stable open set which is affine (see \cite[Lemma~26.21.7]{St}), and contains $m$.
\end{proof}



The proof of the following lemma is omitted.

\begin{lm} \label{Lemma 2}  
Let $X$ and $Y$ be integral varieties over a field $k$.
Let $G$ be a finite group. There
 is an equivalence between:

(i) There is an action of $G$ on the fields $k(X)$ and $k(Y)$,
trivial on $k$,  and there exists an isomorphism of fields $k(X) \simeq k(Y)$
which respects the $G$-action.

(ii) There exist a $G$-stable non-empty open set $U \subset X$
and a  $k$-birational $G$-morphism  $U \to Y$ of $k$-varieties.

(iii) There exist non-empty open sets $U \subset X$ and $V \subset Y$
each with a $G$-action and a $G$-isomorphism $U \simeq V$
of $k$-varieties.
\qed
\end{lm}

 If these properties hold, one says that there is a $G$-birational equivalence between $X$
 and $Y$.


\begin{prop} \label{Proposition 3}
Let $k$ be a field. Let $G$ be an abstract finite group.
Let $X$ and $Y$  be   quasi-projective  geometrically integral  $k$-varieties equipped with a $G$-action.
Suppose there exist
a  non-empty  $G$-stable open set $U \subset X$ and a $G$-morphism of
$k$-varieties $\phi\colon U \to Y$.  Then there exist a normal $k$-variety $ \tilde{X}$
with a $G$-action on $\tilde{X}$,  a  birational $G$-morphism
$p\colon  \tilde{X} \to X$, a non-empty $G$-stable open set $V \subset X$
such that $p\colon p^{-1}(V) \to V$ is an isomorphism
and a $G$-morphism $q\colon \tilde{X} \to Y$
such that $q$ and $\phi \circ p$ coincide on $p^{-1}(V)$.
If $X$ is normal and $Y/k$ is projective, then there exists such
an open set $V$ which contains all codimension 1 points of $X$.
\end{prop}


\begin{proof}
There is a $G$-action on the graph $\Gamma \subset U \times Y $
of the morphism  $U\to  Y$. The projection map $\Gamma \to U$ is an isomorphism.
Let $Z \subset X \times Y$ be the closure of this graph.
It is geometrically integral.
The action of $G$ extends to  $Z \subset X \times Y$.
The projection map $q\colon Z \to X$ is birational.
Over $U$, the map  $q = q^{-1}(U) \to U$ is an isomorphism.

Let $\tilde{Z}$ be the normalization of $Z$.
Since $G$ acts on $Z$, it acts on $\tilde{Z}$.
 Indeed, via Lemma \ref{Lemma 1}  this statement is reduced
 to the following statement.
Suppose we have an inclusion of integral commutative rings $A \subset B$
with same field of fraction $K$, with   $B$ normal and  integral over $A$.
Suppose we have an action of $G$ on $A$, hence on  $K$.
If we are given $g\in G$ and $b \in B$ satisfying an equation
$b^n+ \sum_{i=0}^{n-1} a_{i}b^{i}=0$,
then
$(gb)^n+ \sum_{i=0}^{n-1} a_{i}(gb)^{i}=0$.
Thus $gb \in B$.

One then takes $V$ to be the open set of normal points of $U$.

 If $X$ is normal and $Y/k$ is projective, then the given map
 $U \to Y$ extends to a $G$-morphism from an open set $U \subset X$
 which contains all codimension 1 points of $X$.
\end{proof}


\begin{prop} \label{Proposition 4}
   Let $k$ be a field. Let $G$ be an abstract finite group.
Let $X$, $Y$  be smooth, projective, geometrically integral $k$-varieties
with a $G$-action over $k$. Let $K/k$ be a Galois extension with group $\Gamma$,
finite or infinite. Let $X_K=X\times_kK$, $Y_K=Y\times_kK$ be equipped
with the natural action of $W=G\times \Gamma$.

Suppose there exists a $G$-birational equivalence
between $X$ and $Y$.
Then there exist $W$-permutation modules
$P_{1}$ and $P_{2}$ and a $W$-isomorphism
$$\Pic(X) \oplus P_{1} \simeq \Pic(Y) \oplus P_{2}.$$
\end{prop}


\begin{proof}
Once equipped with Proposition \ref{Proposition 3}, one may
 copy the proof
of  \cite[Prop.~2.A.1, p.~461/462]{CTS2}
(a proof due to Laurent Moret-Bailly) word for word.
One gets isomorphisms of $W$-modules
$$\Pic(X) \oplus P_{1} \simeq  CH^1(\tilde{Z}), $$
$$\Pic(Y) \oplus P_{2} \simeq  CH^1(\tilde{Z}),$$
where $CH^1(\tilde{Z})$ is the Chow group of codimension 1
cycles modulo rational equivalence on the normal variety $\tilde{Z}$ of Proposition \ref{Proposition 3}, and where $P_{1}$
and $P_{2}$ are permutation modules.
This gives a $W$-isomorphism
$$\Pic(X) \oplus P_{1} \simeq \Pic(Y) \oplus P_{2}.$$
\end{proof}


\begin{cl} \label{Corollary 5}
Let $X$, $Y$  be smooth, projective, geometrically integral $k$-varieties
with a $G$-action. Suppose there exists a $G$-birational equivalence
between $X \times \mathbb P(V_1)$ and $Y \times \mathbb P(V_2)$, where $G$ acts
on $\mathbb P(V_1)$ and $\mathbb P(V_2)$ via  homomorphisms $G \to \PGL(V_1)$ and $G\to \PGL(V_2)$, respectively. 
Then  there exist $W$-permutation modules
$P_{1}$ and $P_{2}$ and a $W$-isomorphism
$$\Pic(X_K) \oplus P_{1} \simeq \Pic(Y_K) \oplus P_{2}.$$
\end{cl}

\begin{proof}
The action of $G$ on $\Pic(\mathbb P(V_i)) = \mathbb Z$ ($i=1,2$)
is the trivial one. We have natural $G$-isomorphisms
$$ \Pic(X \times \mathbb P(V_1)) \simeq  \Pic(X) \oplus \Pic(\mathbb P(V_1))  =  \Pic(X) \oplus \mathbb Z,$$
and similarly for $ \Pic(Y \times \mathbb P(V_2))$.  It remains to go over to $X_K$, $Y_K$ and
apply Proposition \ref{Proposition 4}.
\end{proof}


\begin{rk} If $\Char k=0$ and resolution of singularities is available, there is a shorter proof.
One can use weak factorization (which is functorial, over any field of characteristic zero, and in presence of group actions):
every birational map can be factored in a sequence of blowups and blowdowns with centres in smooth subvarieties, this implies
the required statement;
cf. \cite[Proposition~2.5]{BP}.
\end{rk}

Setting $K=\bar k$ in Corollary \ref{Corollary 5}, we arrive at the statement of Theorem \ref{th:crit}. \qed

\section{Real plane} \label{sec:real}

In this section, we show some applications of Theorem \ref{th:crit}. Namely,
it can be viewed as a purely algebraic and computationally feasible alternative
to some more sophisticated tools such as classification of Sarkisov links used
in the study of birational involutions of real plane \cite{CMYZ} and other problems
of a similar flavour. Moreover, in the absence of such a classification for higher-dimensional
varieties,  Theorem \ref{th:crit} may turn out to be the only available tool for
proving the non-linearizability.

Below we present a simple example of such an application.

Consider the following question.

\begin{Qu} \label{q:real-inv}
What are elements of order 2 in $\Cr_2(\mathbb R)$ that are $\mathbb C$-linearizable
but not $\mathbb R$-linearizable?
\end{Qu}

This question was answered in \cite{CMYZ} as a consequence of the classification of
birational involutions of real plane up to conjugacy. According to the Main Theorem in
1.E and its proof in Section 8 (particularly, Lemma 8.2 and Section 8.B), any involution
as in Question \ref{q:real-inv} is conjugate either to one of the two involutions acting
on a quadric of signature (3,1) or (2,2), or to one of the Trepalin involutions
(subdivided into three infinite families), each thereof acting on a conic bundle with
even number (at least 4) of degenerate fibres.

Below we present in some detail an alternative proof of the non-linearizability for
one of these families (as said, the original proof in \cite{CMYZ} is based on subtle
techniques from equivariant birational geometry), providing a little stronger result.

\begin{Ex} \label{Trepalin}
Let $n\ge 1$, and let $Z_n$ denote the surface in $\mathbb P^2_{\mathbb R}\times \mathbb P^1_{\mathbb R}$
given by
\begin{equation} \label{Zn}
x^2\prod_{k=1}^{2n}(t^2_0+k^2t^2_1) + y^2t^{4n}_0 +z^2t^{4n}_1 =0,
\end{equation}
where $([x:y:z], [t_0:t_1])$ denote the coordinates in $\mathbb P^2_{\mathbb R}\times \mathbb P^1_{\mathbb R}$.

This surface was first considered by Trepalin \cite{Tr} and then studied
in \cite{Ya2} and \cite{CMYZ}. We reproduce some of its properties described in
\cite[Example~1.4]{Ya2} (with minor corrections).

$\bullet$ The projection onto $\mathbb P^1_{\mathbb R}$ defines on $Z_n$
a structure of a conic bundle with $r:=4n+2$ degenerate fibres at the points
$p_k=[k\sqrt{-1}:1]$, $p_{-k}=[-k\sqrt{-1}:1]$ ($k=1,\dots ,2n$), $p_0=[0:1]$, $p_{\infty}=[1:0]$.
The set $D$ of the components of the degenerate fibres consists of the lines \{$l_k$, $\bar l_k$\},
\{$l_{-k}$, $\bar l_{-k}$\}, $k=1, \dots ,2n$,
of the form $z=\pm k^{4n}\sqrt{-1}y$ (for the first
$4n$ fibres), \{$l_0$, $\bar l_0$\} given by $z=\pm ((2n)!)^2\sqrt{-1}x$ (for the fibre at $p_0$),
and \{$l_{\infty}$, $\bar l_{\infty}$\} given by $y=\pm \sqrt{-1}x$ (at $p_{\infty}$).
The group $\Gamma=\Gal (\mathbb C/\mathbb R)=\left<\sigma\right>$ acts on $D$ by swapping
$l_{k}$ and $\bar l_{-k}$ ($k=\pm 1, \dots ,\pm 2n$), $l_0$ and $\bar l_0$, $l_{\infty}$ and $\bar l_{\infty}$.
Since the lines in each of the first $4n$ orbits of $\Gamma$ do not intersect, we can blow down the first $4n$
fibres over $\mathbb R$ and obtain a real conic bundle with two degenerate fibres, which is $\mathbb R$-rational
by a theorem of Iskovskikh \cite[Theorem~4.1]{Is}.  Hence $Z_n$ is also $\mathbb R$-rational.

$\bullet$ The involution of $\mathbb P^1_{\mathbb R}$ defined by $[t_0:t_1] \mapsto [-t_0:t_1]$
gives rise to an $\mathbb R$-regular involution $g_n$ of $Z_n$. Denote $G_n=\left<g_n\right>$.
The group $G_n$ acts on $D$
by swapping $l_{k}$ and $l_{-k}$, $\bar l_{k}$ and $\bar l_{-k}$ ($k=1, \dots ,2n$), and fixing the
remaining four lines. Since $Z_n$ is $\mathbb R$-rational, $G_n$ corresponds to a subgroup of order two in
$\Cr_2(\mathbb R)$, for which we will use the same notation.

\begin{Prop} \label{prop:real}
The subgroup $G_n\subset\Cr_2(\mathbb R)$ is $\mathbb C$-linearizable but not stably projectively
$\mathbb R$-linearizable. For $n\ne m$ the subgroups $G_n$ and $G_m$ are not
stably conjugate.

\begin{proof}
Denote $N_n=\Pic (Z_n\times_{\mathbb R}\mathbb C)$. This is a free $\mathbb Z$-module
of rank $r+2$ where $r=4n+2$ is the number of degenerate fibres of the conic bundle.
One can choose as generators of $N$ (the classes of) a section $s$, the general fibre $l$,
and one component of each of the $r$ degenerate fibres (not meeting $s$), say, $l_k$ ($k=\pm 1,\dots ,\pm 2n$),
$l_0$, $l_{\infty}$ (see, e.g. \cite[Proposition~0.4]{KuTs} for more details). As to the
additional component $\bar l_k$ of each degenerate fibre, in $N$ we have $\bar l_k=l-l_k$.

The group $G_n$ acts on $N_n$ by swapping $l_k$ and $l_{-k}$ and fixing all the remaining elements of
the chosen free basis, so that $N$ is a permutation $G_n$-module and is hence flasque and coflasque. Therefore $G_n$ is
$\mathbb C$-linearizable by a theorem of Bogomolov and Prokhorov \cite[Corollary~1.2]{BP}.

The group $\Gamma$ acts on $N$ as follows: $l_k\mapsto l-l_{-k}$ ($k=\pm 1,\dots ,\pm 2n$),
$l_0\mapsto l-l_0$, $l_{\infty}\mapsto l-l_{\infty}$, $l$ is fixed. As to $s$, we shall use
the formula from \cite[page 586]{KuTs}:
$$
2(\sigma (s)-s) = \sum_{i=1}^r (\sigma (L_i)-L_i),
$$
where $L_i$ is a component of the $i^{\textrm th}$ degenerate fibre. In our notation this gives
$$
\sigma (s) = s + (2n+1)l - \sum_{k=1}^{2n}(l_k+l_{-k}) - l_0 -l_{\infty}.
$$

Since $H^1(G_n,N_n)=0$, the restriction-inflation sequence for the group $W_n=G_n\times \Gamma$ gives
$H^1(W_n,N_n)=H^1(G_n,N_n^{\Gamma})$.
We have
$$
\begin{aligned}
N_n^{\Gamma} & =\left<2s - \sum_{k=1}^{2n}(l_k+l_{-k}) - l_0 -l_{\infty}, l, l_1-l_{-1}, \dots ,l_{2n} - l_{-2n}\right> \\
 & =: \left<\tilde s, \tilde l, \tilde l_1, \dots , \tilde l_{2n}\right>.
\end{aligned}
$$

The action of $G_n$ on $N_n^{\Gamma}$ fixes $\tilde s$ and $\tilde l$ and maps each $\tilde l_k$ to $-\tilde l_k$.
Hence we obtain $H^1(W_n,N_n)=(\mathbb Z/2)^{2n}$. Thus the $W_n$-module $N_n$ is not stably permutation and hence
by Theorem \ref{th:main}, we conclude that $G_n$ is not stably projectively $\mathbb R$-linearizable. By the same theorem,
for $n\ne m$ the subgroups $G_n$ and $G_m$ are not stably conjugate in $\Cr_2(\mathbb R)$ because the modules
$N_n$ and $N_m$ are not similar having different first cohomology.
\end{proof}

\end{Prop}

\end{Ex}

\section{Concluding remarks} \label{sec:concl}

We finish by collecting several miscellaneous remarks and questions.

\begin{Rk}
We do not know an answer to Question \ref{q:main} in the case $n=2$,
even for cyclic $G$, even for $G$ of order 2. Hopefully, one can somehow
use the available classification results \cite{DI}, \cite{PSY}.

We do not know anything about the local-global behaviour of elements of odd order
in Cremona groups.
\end{Rk}

\begin{Rk}
Somewhat surprisingly, it is not clear how the linearization properties behave
with respect to Weil's functor of restriction of scalars. In particular,
an affirmative answer to the following question could simplify the argument
at Step 4 of the proof of Theorem \ref{th:main}.
\end{Rk}

\begin{Qu}
Let $L/k$ be a finite field extension.
Suppose that the action of a group $G$ on a smooth, projective $L$-variety $X$
is (projectively) $L$-linearizable. Is the induced $G$-action on $R_{L/k}X$
(projectively) $k$-linearizable?
\end{Qu}

The following partial answer was communicated to me by Evgeny Shinder.

\begin{Prop} (Shinder)
Let $L/k$ be a finite field extension.
Suppose that the action of a finite group $G$ on $X$
is $L$-linearizable. Then the induced $G$-action on $R_{L/k}X$
is stably $k$-linearizable.
\end{Prop}

\begin{proof}
We will use the no-name lemma in the following form \cite[Lemma~4.4]{CGR}:
\begin{lm} \label{lem:noname-CGR}
If $G$ is a finite group acting faithfully on $X$ and $\mathcal E\to X$ is a $G$-linearized
vector bundle of relative dimension $d$, then $\mathcal E$ is $G$-birational to $X\times \mathbb A^d$
with trivial action on the second factor.
\end{lm}

To prove the proposition, we can assume $X = \mathbb P^n_L$ with a $G$-action induced from a linear
$G$-representation on $\mathbb A^{n+1}_L$. Denote $d = [L:k]$. We have an induced $G$-equivariant rational map
$$
\mathbb A_k^{(n+1)d} = R_{L/k}\mathbb A_L^{n+1} \dashrightarrow R_{L/k}Tot(\mathcal O_{\mathbb P_L^n}(-1)) \to  R_{L/k}\mathbb P_L^n.
$$
Here $R_{L/k}Tot(\mathcal O_{\mathbb P_L^n}(-1))$ is a $G$-linearized vector bundle over $R_{L/k}\mathbb P_L^n$.
(Weil's descent of a $G$-linearizable vector bundle is a $G$-linearizable vector bundle, see \cite[Lemma~2.2]{Bri1}.)
Thus by Lemma \ref{lem:noname-CGR} we have a $G$-birational map
$$
R_{L/k}Tot(\mathcal O_{\mathbb P_L^n}(-1)) \dashrightarrow R_{L/k}\mathbb P_L^n \times \mathbb A_k^d.
$$
Therefore $R_{L/k}\mathbb P_L^n$ is stably $k$-linearizable.
\end{proof}

\medskip

\noindent{\it Acknowledgement.} I thank Mikhail Borovoi, Ivan Cheltsov, Mathieu Florence, Yuval Ginosar, Evgeny Shinder,
Yuri Tschinkel, and Michael Tsfasman for useful discussions and correspondence. My special thanks go to Jean-Louis
Colliot-Th\'el\`ene whom I owe Section \ref{sec:crit} as well as numerous critical remarks on the earlier versions 
of this text. I am also grateful to the referees for careful reading and helpful comments.

\enddocument